\documentclass{amsart}

\input xy
\xyoption{all}

\usepackage{geometry}
\usepackage{amsmath}
\usepackage{amssymb}
\usepackage{cite}
\usepackage{amsthm}  
\usepackage{tikz}
\usetikzlibrary{er,positioning}

\newtheorem{same}{This should never appear}[section]
\newtheorem{defin}[same]{Definition}

\newtheorem{remark}[same]{Remark}
\newtheorem{theorem}[same]{Theorem}

\newtheorem{lemma}[same]{Lemma}
\newtheorem{fact}[same]{Fact}

\newtheorem{cor}[same]{Corollary}

\newtheorem{hypothesis}[same]{Hypothesis}
\newtheorem{nota}[same]{Notation}

\newtheorem{defin*}{Definition}
\newtheorem*{theorem*}{Theorem}

\makeatletter
\newcommand{\skipitems}[1]{%
  \addtocounter{\@enumctr}{#1}%
}
\newcommand{\bb}{\mathbf{b}}

\newcommand{\id}{\textrm{id}}
\newcommand{\eff}{\mathcal{F}}

\newcommand{\K}{\mathbf{K}}

\newcommand{\Kf}{\K^{\eff}}

\newcommand{\LS}{\operatorname{LS}}

\newcommand{\leap}[1]{\le_{#1}}

\newcommand{\lea}{\leap{\K}}

\newcommand{\gtp}{\mathbf{gtp}}
\newcommand{\gS}{\mathbf{gS}}

\DeclareMathOperator{\pp}{pp}    
\DeclareMathOperator{\cof}{cf}    

\title{On superstability in the class of flat modules and perfect rings}
\date{\today.} 

\author{Marcos Mazari-Armida}
\email{mmazaria@andrew.cmu.edu}
\urladdr{http://www.math.cmu.edu/~mmazaria/ }
\address{Department of Mathematical Sciences \\ Carnegie Mellon
University \\ Pittsburgh, Pennsylvania, USA}

\begin{document}

\begin{abstract}

We obtain a characterization of  left perfect rings via superstability of the class of flat left modules with pure embeddings.

\begin{theorem} For a ring $R$ the following are equivalent.
\begin{enumerate}
\item $R$ is left perfect.
\item The class of flat left $R$-modules with pure embeddings is superstable.
\item There exists a $\lambda \geq (|R| + \aleph_0)^+$ such that the class of flat left $R$-modules with pure embeddings has uniqueness of limit models of cardinality $\lambda$.
\item Every limit model in the class of flat left $R$-modules with pure embeddings is $\Sigma$-cotorsion.
\end{enumerate}
\end{theorem}

A key step in our argument is the study of limit models in the class of flat modules. We show that  limit models with chains of long cofinality are cotorsion and that limit models are elementarily equivalent. 

We obtain a new characterization via limit models of the rings characterized in \cite{roth}. We show that in these rings the equivalence between  left perfect rings and superstability can be refined. We show that the results for these rings can be applied to extend \cite[1.2]{sh820} to classes of flat modules not axiomatizable in first-order logic.

\end{abstract}


\maketitle

{\let\thefootnote\relax\footnote{{AMS 2010 Subject Classification:
Primary: 03C48, 16B70. Secondary: 03C45, 03C60, 13L05, 16L30, 16D10
}}}

\tableofcontents

\section{Introduction}

An \emph{abstract elementary class} (AEC for short) is a pair $\K=(K, \lea)$, where $K$ is class of structures  and $\lea$ is a partial order on $K$ extending the substructure relation. AECs are closed under directed limits and every subset of a model in the class is contained in a small model in the class. Shelah introduced them in \cite{sh88} to capture the semantic structure of non-first-order theories.  Some interesting algebraic examples are: abelian groups with embeddings, torsion-free groups with pure embeddings, $R$-modules with embeddings, $R$-modules with pure embeddings and first-order axiomatizable classes of modules with pure embeddings. In this paper, we focus on the class of flat modules with pure embeddings. This is an AEC because flat modules are closed under pure submodules and directed limits. This class was already considered in \cite[\S 6]{lrv}.

\emph{Superstable theories} were introduced by Shelah in \cite{sh1} as part of his project to find dividing lines on the class of  complete first-order theories. This project is still central in current research in Model Theory. For AECs, Shelah introduced superstability in \cite{sh394}. Until recently it was believed to suffer from ``schizophrenia" \cite[p. 19]{shelahaecbook}, since it was not known if many natural conditions that were believed to characterize superstability were equivalent.  In \cite[1.3]{grva} and \cite{vaseyt}, it was shown (under extra hypotheses that are satisfied by the class of flat modules\footnote{The hypotheses are amalgamation, joint embedding, no maximal models and tameness.})  that superstability is a well-behaved concept as they showed that many possible characterizations of superstability are equivalent. Due to this and the important role that limit models play in this paper, we say that an AEC is \emph{superstable} if it has uniqueness of limit models on a tail of cardinals.\footnote{For a complete first-order $T$, $(Mod(T), \preceq)$ is superstable if and only if $T$ is superstable as a first-order theory, i.e., $T$ is $\lambda$-stable for every $\lambda \geq 2^{|T|}$.} Recall that a \emph{limit model} is a universal model with some level of homogeneity (see Definition \ref{limit}).  

 A ring $R$ is \emph{left perfect} if every flat left $R$-module is a projective module. Left perfect rings were introduced by Bass in \cite{bass}. They play a significant role in homological algebra (see \cite[\S 8]{lam}). Xu was the first to notice a relation between perfect rings and cotorsion modules in \cite[3.3.1]{xu}.

In this paper, we provide further evidence that the concept of superstability has algebraic significance. In the context of AECs this was first noticed in \cite{maz1}. Prior to it, there were a few papers \cite{shlaz}, \cite{balmc} and \cite{grsh} where notherian rings, artinian rings and superstability were related.

More precisely, we characterize left perfect rings via superstability of the class of flat left modules with pure embeddings. The main theorem of the paper is the following.

\textbf{Theorem \ref{main3}.} \textit{ For a ring $R$ the following are equivalent.
\begin{enumerate}
\item $R$ is left perfect.
\item The class of flat left $R$-modules with pure embeddings is superstable.
\item There exists a $\lambda \geq (|R| + \aleph_0)^+$ such that the class of flat left $R$-modules with pure embeddings has uniqueness of limit models of cardinality $\lambda$.
\item Every limit model in the class of flat left $R$-modules with pure embeddings is $\Sigma$-cotorsion.
\end{enumerate}}

In order to obtain the above equivalence, we study the limit models in the class of flat modules. We show that limit models with chains of long cofinality are cotorsion (Theorem \ref{bigpi2}), show that limit models are elementarily equivalent (Lemma \ref{elem}) and characterize limit models of countable cofinality (Lemma \ref{ccountablelim}).

Merging the main theorem of this paper together with the characterization of noetherian rings via superstability obtained in \cite[3.12]{maz1}; we obtain a characterization of artinian rings via superstability (Corollary \ref{art}).


In contrast to previous results on limit models with chains of long cofinality ( \cite[4.10]{maz}, \cite[4.5]{kuma} ), limit models with chains of long cofinalities in this case might not be pure-injective. This happens precisely because the class of flat modules is not necessarily closed under pure-injective envelopes. We obtain the following.

\textbf{Theorem \ref{newroth}.} \textit{  For a ring $R$ the following are equivalent.
\begin{enumerate}
\item Every $(\lambda, \alpha)$-limit model in the class of flat modules (with pure embeddings) with $\lambda \geq  (|R| + \aleph_0)^+$ and $\cof(\alpha)\geq (|R| + \aleph_0)^+$ is  pure-injective.
\item The pure-injective envelope of every flat left $R$-module is flat.
\end{enumerate}}

Since the rings characterized in \cite{roth} are those that satisfy the second condition of the above theorem,  the result gives a new characterization of such rings. In these rings we characterize the Galois-types and the stability cardinals of the class of flat modules with pure embeddings. As a simple corollary we obtain a result of Shelah regarding universal torsion-free abelian groups with respect to pure embeddings \cite[1.2]{sh820} (see Lemma \ref{g-st} and the remark below it). Moreover, by using that flat cotorsion modules are the same as pure-injective modules in this special case, we are able to  lower the bound in Theorem \ref{main3} where the tail of cardinals where uniqueness of limit models begins to $|R|+ \aleph_0$ (Theorem \ref{main1}). 

The class of flat modules is not first-order order axiomatizable (\cite[Theo. 4]{eksa}), but it is axiomatizable in $\mathbb{L}_{\infty, \omega}$ (\cite[\S 2]{hero}). Due to this, the results of this paper lie outside of the scope of first-order model theory and hint to the importance of the development of non-first-order methods.

The paper is divided into four sections. Section 2 presents necessary background. Section 3  studies limit models in the class of flat modules with pure embeddings and provides a new characterization of left perfect rings via superstability of the class of flat modules. Section 4 studies limit models in the class of flat modules under an additional assumption, characterizes this assumption via limit models and provides a refinement of the main theorem under the additional hypothesis.

This paper was written while the author was working on a Ph.D. under the direction of Rami Grossberg at Carnegie Mellon University and I would like to thank Professor Grossberg for his guidance and assistance in my research in general and in this work in particular. I would like to thank Thomas G. Kucera for sharing his module theoretic knowledge.  I would like to thank John T. Baldwin, Philipp Rothmaler and the referee for many valuable comments that significantly improved the paper.


\section{Preliminaries}

We present the basic concepts of abstract elementary classes that are used in this paper. These are further studied in \cite[\S 4 - 8]{baldwinbook09} and  \cite[\S 2, \S 4.4]{ramibook}.  Regarding the background on module theory, we give a brief survey of the concepts we will use in this paper and present a few concepts throughout the text. The main module theoretic ideas used in this paper are studied in detail in \cite{xu}.

\subsection{Abstract elementary classes}
Abstract elementary classes (AECs) were introduced by Shelah in
\cite[1.2]{sh88}. Among the requirements we have that an AEC is closed under directed limits and that every set is contained in a small model in the class. The reader can consult the definition in \cite[4.1]{baldwinbook09}.  
\begin{nota}\
\begin{itemize}
\item Given a model $M$, we will write $|M|$ for its underlying set and $\| M \|$ for its cardinality. 

\item If $\lambda$ is a cardinal and $\K$ is an AEC, then $\K_{\lambda}=\{ M \in \K : \| M \|=\lambda \}$.

\item Let $M, N \in \K$. If we write ``$f: M \to N$", we assume that
$f$ is a $\K$-embedding, i.e., $f: M \cong f[M]$ and $f[M] \lea N$.
In particular, $\K$-embeddings are always monomorphisms.
\end{itemize}
\end{nota}

In \cite{sh300} Shelah introduced a notion of semantic type. The
original definition was refined and extended by many authors who
following \cite{grossberg2002} call
these semantic types Galois-types (Shelah recently named them orbital
types).
We present here the modern definition and call them Galois-types
throughout the text. We follow the notation of \cite[2.5]{mv}.

\begin{defin}\label{gtp-def}
  Let $\K$ be an AEC.
  
  \begin{enumerate}
    \item Let $\K^3$ be the set of triples of the form $(\bb, A,
N)$, where $N \in \K$, $A \subseteq |N|$, and $\bb$ is a sequence
of elements from $N$. 
    \item For $(\bb_1, A_1, N_1), (\bb_2, A_2, N_2) \in \K^3$, we
say $(\bb_1, A_1, N_1)E_{\text{at}}^{\K} (\bb_2, A_2, N_2)$ if $A
:= A_1 = A_2$, and there exist $\K$-embeddings $f_\ell : N_\ell \to_A N$ for $\ell \in \{ 1, 2\}$ such that
$f_1 (\bb_1) = f_2 (\bb_2)$ and $N \in \K$.
    \item Note that $E_{\text{at}}^{\K}$ is a symmetric and
reflexive relation on $\K^3$. We let $E^{\K}$ be the transitive
closure of $E_{\text{at}}^{\K}$.
    \item For $(\bb, A, N) \in \K^3$, let $\gtp_{\K} (\bb / A;
N) := [(\bb, A, N)]_{E^{\K}}$. We call such an equivalence class a
\emph{Galois-type}. Usually, $\K$ will be clear from the context and we will omit it.
\item For $M \in \K$, $\gS_{\K}(M)= \{  \gtp_{\K}(b / M; N) : M
\leq_{\K} N\in \K \text{ and } b \in N\} $ 
\item For $\gtp_{\K} (\bb / A; N)$ and $C \subseteq A$, $\gtp_{\K} (\bb / A; N)\upharpoonright_{C}:= [(\bb, C, N)]_E$.

  \end{enumerate}

\end{defin}

\begin{defin} An AEC $\K$ is \emph{$\lambda$-stable}  if for any $M \in
\K_\lambda$, $| \gS_{\K}(M) | \leq \lambda$. 
\end{defin}

Recall the following notion that was isolated by
Grossberg and VanDieren in \cite{tamenessone}.

\begin{defin} 
$\K$ is \emph{$(< \kappa)$-tame} if for any $M \in \K$ and $p \neq q \in \gS(M)$,  there is $A \subseteq |M|$ such that $|A |< \kappa$ and $p\upharpoonright_{A} \neq q\upharpoonright_{A}$.
\end{defin}

Before introducing the concept of limit model we recall the concept of universal extension.

\begin{defin}
$M$ is $\lambda$-\emph{universal over} $N$ if and only if $N \lea M$
 and for any $N^* \in \K_{\leq\lambda}$ such that
$N \lea N^*$, there is $f: N^* \xrightarrow[N]{} M$. $M$ is \emph{universal
over} $N$ if and only if $\| N\|= \| M\| $ and $M$ is $\| M
\|$-\emph{universal over} $N$. 
\end{defin}

With this we are ready to introduce limit models, they were originally introduced in \cite{kosh}.

\begin{defin}\label{limit}
Let $\lambda$ be an infinite cardinal and $\alpha < \lambda^+$ be a limit ordinal.  $M$ is a \emph{$(\lambda,
\alpha)$-limit model over} $N$ if and only if there is $\{ M_i : i <
\alpha\}\subseteq \K_\lambda$ an increasing continuous chain such
that $M_0 :=N$, $M_{i+1}$ is universal over $M_i$ for each $i <
\alpha$ and $M= \bigcup_{i < \alpha} M_i$.

$M$ is a $(\lambda, \alpha)$-limit model if there is $N \in
\K_\lambda$ such that $M$ is a $(\lambda, \alpha)$-limit model over
$N$. $M$ is a $\lambda$-limit model if there is a limit ordinal
$\alpha < \lambda^+$ such that $M$  is a $(\lambda,
\alpha)$-limit model. We say that $M$ is a limit model if there is an infinite cardinal $\lambda$ such that $M$ is a $\lambda$-limit model.

\end{defin}
\begin{flushright}
•
\end{flushright}

Observe that if $M$ is a $\lambda$-limit model, then $M$ has cardinality $\lambda$. The next fact gives conditions for the existence of limit models.

\begin{fact}[{\cite[\S II]{shelahaecbook}, \cite[2.9]{tamenessone}}]\label{existence}
Let $\K$ be an AEC with joint embedding, amalgamation and no maximal models. If $\K$ is $\lambda$-stable, then for every $N \in \K_\lambda$ and limit ordinal $\alpha < \lambda^+$ there is $M$ a $(\lambda, \alpha)$-limit model over $N$. Conversely, if $\K$ has a $\lambda$-limit model, then $\K$ is $\lambda$-stable
\end{fact}

The key question regarding limit models is the uniqueness of $\lambda$-limit models for a fixed cardinal $\lambda$. When the lengths of the cofinalities of the chains of the limit models are equal, one can show that the limit models are isomorphic by a back-and-forth argument.\footnote{Hence, for a fixed cardinal $\lambda$ and a fixed limit ordinal $\alpha < \lambda^+$, there is a unique $(\lambda, \alpha)$-limit model.} Therefore, the question is what happens when the cofinalities of the chains of the limit models are different. This has been studied thoroughly in the context of abstract elementary classes \cite{shvi}, \cite{van06}, \cite{grvavi}, \cite{extendingframes}, \cite{vand}, \cite{bovan} and \cite{vasey18}.

\begin{defin}
$\K$ has \emph{uniqueness of limit models of cardinality $\lambda$} if $\K$ has $\lambda$-limit models and if given $M, N$ $\lambda$-limit models, $M$ and $N$ are isomorphic.
\end{defin}

In \cite[1.3]{grva} and \cite{vaseyt} it was shown that for AECs that have amalgamation, joint embedding, no maximal models and are tame, the definition below is equivalent to every other definition of superstability considered in the context of AECs.  Since the class of flat modules with pure embeddings satisfies these properties  (see Fact \ref{ffact}), we introduce the following as the definition of superstability.

\begin{defin}
$\K$ is a \emph{superstable} AEC if and only if $\K$ has uniqueness of limit models on a tail of cardinals.
\end{defin}

\begin{remark} For a complete first-order $T$, $(Mod(T), \preceq)$ is superstable if and only if $T$ is superstable as a first-order theory, i.e., $T$ is $\lambda$-stable for every $\lambda \geq 2^{|T|}$. The forward direction follows from Fact \ref{existence} and the backward direction from \cite[1.6]{grvavi}.
\end{remark}

Finally, recall the standard notion of a universal model.

\begin{defin}
Let $\K$ be an AEC and $\lambda$ be a cardinal. $M \in \K$ is a \emph{universal model in
$\K_\lambda$} if $M \in \K_\lambda$ and if given any $N \in \K_\lambda$, there is $f: N \to M$.
\end{defin}

The following fact will be useful.

\begin{fact}[{\cite[2.10]{maz}}]\label{euni}
Let $\K$ be an AEC with the joint embedding property. If $M$ is a $\lambda$-limit model, then $M$ is a universal model in $\K_\lambda$.
\end{fact}

\subsection{Module Theory}  All rings considered in this paper are associative with an identity element. Recall that a left $R$-module $F$ is \emph{flat} if $(-) \otimes F$ is an exact functor. $M$ is a \emph{pure submodule} of $N$, denoted by $\leq_{pp}$, if for every $L$ right $R$-module $L \otimes M \to L \otimes N$ is a monomorphism.

\begin{nota} Given a ring $R$, let $\K^{\eff}=( K_{flat}, \leq_{pp})$ where $K_{flat}$ is the class of flat left $R$-modules  and $\leq_{pp}$ denotes the pure submodule relation. \end{nota}

We assume the reader is familiar with pure-injective modules (see for example \cite[\S 2]{prest}) and focus on cotorsion  modules. Cotorsion modules were introduced by Harrison in \cite{har}.

\begin{defin}
A left $R$-module $M$ is \emph{cotorsion} if and only if $Ext^1(F, M)=0$ for every flat module $F$, or equivalently, every short exact sequence $0 \to M \to N \to F \to 0$  with $F$ a flat module splits.
\end{defin}

It is easy to check that a pure-injective module is cotorsion. The following generalization of Bumby's result \cite{bumby} will be useful.

\begin{fact}[{ \cite[3.2]{gks}}]\label{ipi} Let $M, N$ be cotorsion modules. If there are $f: M \to N$ a pure embedding and $g: N \to M$ a 
pure embedding, then $M$ and $N$ are isomorphic. 

\end{fact}

Similar to the notion of pure-injective envelope there is the notion of cotorsion envelope. These are thoroughly studied by Xu in  \cite[\S 1, 3.4]{xu}.

\begin{defin}
Let $M$ be a module, $M \hookrightarrow_i C(M)$ is the \emph{cotorsion envelope} of $M$ if and only if
\begin{enumerate}
\item If $\phi: M \to C$ and $C$ is a cotorsion module, then there is $f: C(M) \to C$ such that $\phi=f \circ i$.
\item If an endomorphism $f: C(M) \to C(M)$ is such that $i = f\circ i$, then $f$ is an automorphism.
\end{enumerate}
\end{defin}

The existence of a cotorsion envelope for every module is a deep result that is equivalent to \emph{the Flat Cover Conjecture}. The Flat Cover Conjecture was asserted by Enoch in \cite{enoch} and proved twenty years later by Bican,  El Bashir and  Enochs in \cite{bee}. We will use that there are cotorsion envelopes a few times in the text.

An easy assertion that we will use is the following.

\begin{fact}[{\cite[3.4.2]{xu}}]\label{xu1}
If $M \hookrightarrow_i C(M)$ is a cotorsion envelope, then $C(M)/M$ is flat and $M \leq_{pp} C(M)$. Moreover, if $M$ is flat, then $C(M)$ is flat. 
\end{fact}

We will also work with the following class of modules.

\begin{defin}
 A left $R$-module $M$ is \emph{$\Sigma$-cotorsion} if and only if $M^{(I)}$ is a cotorsion module for every index set $I$.
\end{defin}

 Left perfect rings were introduced by Bass in \cite{bass}. They play a significant role in homological algebra and have been  thoroughly studied, see for example \cite[\S 8]{lam}.

\begin{defin}
A ring $R$ is \emph{left perfect} if every flat left $R$-module is a projective module.
\end{defin}

Below we give some equivalent conditions that characterize left perfect rings. Further equivalent conditions that mention cotorsion modules are given in \cite{gh3}.

\begin{fact}[{\cite[3.3.1]{xu}}]\label{xu2} For a ring $R$ the following are equivalent.
\begin{enumerate}
\item $R$ is left perfect.
\item Every flat left $R$-modules is cotorsion.
\item Every left $R$-module is cotorsion.
\end{enumerate}

\end{fact}

\section{The main case}

In this section, we will work with the class of flat modules with pure embeddings. We obtain a characterization of  left perfect rings via superstability of the class of flat left modules with pure embeddings (Theorem \ref{main3}). This is obtained by understanding the limit models of the class.

 The next result follows from \cite[\S 6]{lrv}.

\begin{fact}\label{ffact} Let $R$ be ring and $\K^{\eff}=( K_{flat}, \leq_{pp})$ where $K_{flat}$ is the class of flat left $R$-modules.
\begin{enumerate}
\item $\Kf$ is an AEC with $\LS(\Kf)= |R| + \aleph_0$.
\item $\Kf$ has joint embedding, amalgamation and no maximal models.
\item There is a cardinal $\theta_0\geq |R| + \aleph_0$ such that if $\lambda ^{\theta_0}=\lambda$, then $\Kf$ is $\lambda$-stable.
\item There is a cardinal $\theta_1 \geq |R| + \aleph_0$ such that $\Kf$ is $\theta_1$-tame for Galois-types of finite length.
\end{enumerate}
\end{fact}
\begin{proof}
Observe that in the class of flat modules, $N/M$ is flat if and only if $M$ is a pure submodule of $N$ (see for example \cite[11.1]{ste}). Therefore we can use the results obtained in \cite[\S 6]{lrv}. Since $\Kf$ is a flat-like category in the sense of \cite[6.11]{lrv}  and $\Kf$ is closed under pure submodules, then by \cite[6.20, 6.21]{lrv} it follows that $\Kf$ is an AEC with amalgamation and moreover it has a stable independence notion. Then by \cite[8.16]{lrv1} it follows that (3) and (4) hold. That joint embedding and no maximal models hold, follows from the fact that flat modules are closed under direct sums. 
\end{proof}

\begin{nota} Fix $\theta_0$ and $\theta_1$ as the least cardinals such that (3) and (4) of Fact \ref{ffact} hold. 

\end{nota}

A natural question to ask is the value of $\theta_0$ and $\theta_1$. We say more about this in the next section under additional assumptions, but for now we focus on proving the main theorem of the paper (Theorem \ref{main3}).

Since $\Kf$ has joint embedding, amalgamation and no maximal models, from Fact \ref{ffact} and Fact \ref{existence}, it follows
that $\Kf$ has a $(\lambda, \alpha)$-limit model if $\lambda^{\theta_0}=\lambda$ and $\alpha < \lambda^+$ is a limit ordinal. As hinted by previous results \cite{maz}, \cite{kuma}, limit models with chains of long cofinality are easier to understand than limit models with chains of small cofinality so we study these first. 

Before we characterize these limit models, we need to carefully work with some of the ideas of \cite{gh1} and \cite{gh2}. Recall the following definition.

\begin{defin}[{\cite[Def. 1]{gh1}}]
Let $I$ be a directed system, $(n_i)_{i \in I} \in \mathbb{N}^I$, $\bar{A}=(A_{ij})_{i\leq j}$ with $A_{ij}$ a $n_i \times n_j$ matrix with coefficients in $R$ and $(\bar{x}_i)_{i \in I}$ with $\bar{x}_i$ an $n_i$-tuple of variables. Given $M$ a left $R$-module and $\bb = (\bb_{ij})_{i \leq j} \in \Pi_{i \leq j} M^{n_i}$, we associate the system:
\[ \Omega^{\bar{A}}_{\bb} (\bar{x}_i)_{i\in I} := \{ \bar{x}_i - A_{ij} \bar{x}_j = \bb_{ij} \}_{i \leq j} .\]

We call a system of linear equations \emph{divisible of size $\lambda$} if it is of the form $\Omega^{\bar{A}}_{\bb} (\bar{x}_i)_{i\in I}$, for every $i \leq j \leq k$ it holds that $A_{jk} A_{ij}= A_{ik}$, for every $i \in I$ it holds that $A_{ii}=\id$ and $|I| =\lambda$. 

\end{defin}

The next assertion is a minor improvement of \cite[Cor. 3]{gh1}.

\begin{fact}\label{fcot}
A left $R$-module $M$  is cotorsion if and only if every finitely solvable divisible system of linear equations in $M$ of size at most $|R| + \aleph_0$ is solvable in $M$.
\end{fact}
\begin{proof}
The forward direction is \cite[Cor. 3]{gh1}. For the backward direction, recall that to show that $M$ is cotorsion it is enough to show, by \cite[Prop. 2]{bee}, that $Ext^1(F, M)= 0$ for every flat module $F$ of cardinality at most $|R| + \aleph_0$. Then remember that in Lazard's Theorem the index set $I$ to get a flat module $F$ as a direct limit of finitely generated free $R$-modules is contained in $\{ (J, N) : J \subseteq_{\text{fin}} F \times \mathbb{Z} \text{ and } N \leq R^{( F \times \mathbb{Z})}  \text{ finitely generated} \}$ (see for example \cite[ 8.16]{osb}) and this set has size at most $|R| + \aleph_0$ if $|F| \leq |R| + \aleph_0$. Then by repeating the argument given in \cite[p. 3,4]{gh1} one can obtain the result.
\end{proof}

With this we are able to characterize limit models of big cofinality.

\begin{theorem}\label{bigpi2}
Assume $\lambda\geq (|R| + \aleph_0)^+$. If $M$ is a $(\lambda,
\alpha)$-limit model in $\Kf$ and $\cof(\alpha)\geq (|R| + \aleph_0)^+$, then $M$ is a cotorsion module.
\end{theorem}
\begin{proof} Fix $\{ M_\beta :\beta < \alpha\}$ a witness to the fact that $M$ is a
$(\lambda, \alpha)$-limit model. By Fact \ref{fcot} it is enough to show that every finitely solvable divisible system of linear equations in $M$ of size at most $|R| + \aleph_0$ is solvable in $M$. Let $\Omega^{\bar{A}}_{\bb} (\bar{x}_i)_{i\in I}$ be a divisible system of linear equations satisfying these hypotheses.

Consider $C(M)$ the cotorsion envelope of $M$, observe that $\Omega^{\bar{A}}_{\bb} (\bar{x}_i)_{i\in I}$ is finitely solvable in $C(M)$ because $M \leq_{pp} C(M)$. Since $C(M)$ is cotorsion, by Fact \ref{fcot} $\Omega^{\bar{A}}_{\bb} (\bar{x}_i)_{i\in I}$ is solvable in $C(M)$. Let $(\bar{c}_i)_{i \in I} \in \Pi_{i} C(M)^{n_i}$ be a solution.

Since $\cof(\alpha) \geq (|R| + \aleph_0)^+$ and $|I| \leq |R| + \aleph_0$, there is an ordinal $\beta< \alpha$ such that  $\{ \bb_{ij} : i \leq j \} \subseteq M_\beta$. Observe that $C(M) \in \Kf$ and $M_\beta \leq_{pp} C(M)$  by Fact \ref{xu1}. Then applying the downward L\"{o}wenheim-Skolem-Tarski axiom to $M_\beta \cup \{ \bar{c}_i : i \in I \}$ in $C(M)$ we obtain $M^* \in \Kf_\lambda$ such that $M_\beta \leq_{pp} M^*$ and $\{ \bar{c}_{i} : i \in I \} \subseteq  M^*$. Then there is $f: M^* \xrightarrow[M_\beta]{} M_{\beta+1}$, because  $M_{\beta+1}$ is universal over
$M_\beta$. Since $\{ \bb_{ij} : i \leq j \} $ is fixed by the
choice of $M_\beta$, it is easy to see that $\{ f(\bar{c}_i) : i \in I \}  \subseteq  M_{\beta+1}\leq_{pp}
M$ is a solution to $\Omega^{\bar{A}}_{\bb} (\bar{x}_i)_{i\in I}$ in $M$. Therefore, $M$ is cotorsion. \end{proof}

\begin{remark}
The reader might wonder if the limit model above is pure-injective instead of just cotorsion. This is not the case as the class of flat modules is not necessarily closed under pure-injective envelopes. We will study this with more detail in the next section.
\end{remark}

Since limit models are universal models by Fact \ref{euni}, the following follows from Fact \ref{ipi}.

\begin{cor}\label{easy} Let $\lambda$ be a cardinal. If $M, N$ are $\lambda$-limit models in $\Kf$ and cotorsion modules, then $M$ and $N$ are isomorphic.
\end{cor} 

Putting together the last two assertions we obtain.

\begin{cor}\label{uniqb} Assume $\lambda\geq (|R| + \aleph_0)^+$.
If $M \in \Kf$ is a $(\lambda, \alpha)$-limit model and $N \in \Kf$ is a $(\lambda,
\beta)$-limit model such that $\cof(\alpha), \cof(\beta) \geq (|R| + \aleph_0)^+$,
then $M$ and  $N$ are isomorphic. 
\end{cor}

\begin{remark}  Conjecture 2 of \cite{bovan} asserts that for an AEC $\K$ and $\lambda \geq \LS(\K)$ a regular cardinal such that $\K$ is $\lambda$-stable, the regular ordinals $\alpha$'s  less  than $\lambda^+ $ such that the  $(\lambda, \alpha)$-limit model is isomorphic to the $(\lambda,\lambda)$-limit model is an end segment of regular cardinals. Observe that the above corollary shows that the conjecture is true for $\Kf$ if $R$ is a countable ring and $\lambda$ is a regular uncountable cardinal.
\end{remark}

Recall that two models are elementarily equivalent if they satisfy the same first-order sentences.  Surprisingly, one can still obtain that every two limit models in $\Kf$ are elementarily equivalent. The proof is basically the same as that of \cite[4.3]{kuma} so we omit it.

\begin{lemma}\label{elem}
If $M, N$ are limit models in $\Kf$, then $M$ and $N$ are elementarily equivalent.
\end{lemma}

We do not think that the previous result is as fundamental as the same result for classes axiomatized in first-order logic, see \cite[\S 4.1]{maz1}, but any how this will be useful when characterizing left perfect rings.

The next step will be to characterize limit models of countable cofinality. In order to do that we will need the following remark.

\begin{remark}\label{easyprop}
If $M, N$ are flat modules, $M$ is a cotorsion module and $M\leq_{pp} N$, then $M$ is a direct summand of $N$. This follows from the fact that $N/M$ is a flat module and the definition of cotorsion module.
\end{remark}

Using the above remark together with the fact that flat modules are closed under pure submodules and that cotorsion modules are closed under finite direct sums, we can construct universal extensions and characterize limit models of countable cofinality as in \cite[4.8, 4.9]{kuma}.

\begin{lemma}\label{uextension} Let $\lambda$ be a cardinal.
If $M \in \Kf_{\lambda}$ is cotorsion and $U \in
\Kf_{\lambda}$ is a universal model in $\Kf_{\lambda}$, then $M
\oplus U$ is universal over $M$.
\end{lemma}
\begin{proof}
Let $N \in \Kf_\lambda$ with $M\leq_{pp} N$. By Remark \ref{easyprop} there is an $M'$ such that $N= M \oplus M'$. Since $\Kf$ is closed under pure submodules $M' \in \Kf_{\leq \lambda}$ and by universality of $U$ there is $f: M' \to U$. Then observe that $g: M\oplus M' \to M \oplus U$ given by $g(m + m')= m + f(m')$ is as required. \end{proof}

\begin{lemma}\label{ccountablelim} Assume $\lambda\geq
(|R| + \aleph_0)^+$. If $M \in \Kf$ is
a $(\lambda, \omega)$-limit model and $N \in \Kf$ is a
$(\lambda, (|R| + \aleph_0)^+)$-limit model, then $M$ and  $N^{(\aleph_0)}$ are isomorphic.
\end{lemma} 
\begin{proof}
Let $N$ be a $(\lambda, (|R| + \aleph_0)^+)$-limit model. By Theorem \ref{bigpi2} $N$ is a cotorsion module. Then using the above lemma $\{ N^i : 0 < i < \omega \}$ is a witness to the fact that $N^{(\aleph_0)}$ is a $(\lambda, \omega)$-limit model. Therefore, $M$ is isomorphic to $N^{(\aleph_0)}$ .
\end{proof}

Let us recall the following results from \cite{sast}. They extended to uncountable rings the results of \cite{gh2}. 

\begin{fact}\label{sigmac} Let $R$ be a ring.
\begin{enumerate}
\item If $N$ is $\Sigma$-cotorsion and $M \leq_{pp} N$, then $M$ is $\Sigma$-cotorsion.
\item If $N$ is $\Sigma$-cotorsion and $M$ is elementarily equivalent to $N$, then $M$ is $\Sigma$-cotorsion.
\item  (\cite[3.8]{sast}) $M$ is $\Sigma$-cotorsion if and only if $M^{(|R|+\aleph_0)}$ is a cotorsion module.
\end{enumerate}
\end{fact}
\begin{proof} (1) and (2) follow from \cite[3.3]{sast} and using that the definable subcategory generated by a module is closed under pure submodules and elementarily equivalent modules. 
\end{proof}

The next theorem is the main theorem of the paper. 

\begin{theorem}\label{main3} For a ring $R$ the following are equivalent.
\begin{enumerate}
\item $R$ is left perfect.
\item The class of flat left $R$-modules with pure embeddings is superstable.
\item There exists a $\lambda \geq (|R| + \aleph_0)^+$ such that the class of flat left $R$-modules with pure embeddings has uniqueness of limit models of cardinality $\lambda$.
\item Every limit model in the class of flat left $R$-modules with pure embeddings is $\Sigma$-cotorsion.
\end{enumerate}
\end{theorem}
\begin{proof}
$(1) \Rightarrow (2)$   By Fact \ref{ffact}.(3) there is a $\theta_0 \geq |R| + \aleph_0$ such that $\Kf$ is $\lambda$-stable if $\lambda^{\theta_0}=\lambda$. Let $\lambda_0$ be the least $\lambda$ such that $\Kf$ is $\lambda_0$-stable, we claim that for every $\lambda \geq \lambda_0$, $\Kf$ has uniqueness of limit models of size $\lambda$.

By Fact \ref{xu2} every flat module is a cotorsion module. Then by Corollary \ref{easy} there is at most one $\lambda$-limit model for each $\lambda$ up to isomorphisms. To finish the proof, we show by induction that  for every $\lambda \geq \lambda_0$, $\Kf$ is $\lambda$-stable.

The base step follows from the choice of $\lambda_0$, so we do the induction step. 

Suppose $\lambda$ is an infinite cardinal and that $\Kf$ is $\mu$-stable for every $\mu \in [\lambda_0, \lambda)$. Let $\cof(\lambda)=\kappa$ and $\{\lambda _i  : i < \kappa \}$ be a continuous increasing sequence of cardinals such that $\lambda_i < \lambda$ for each $i < \kappa$ and $sup_{i < \kappa} \lambda_i=\lambda^{-}$.\footnote{For $\theta$ a cardinal, we define $\theta^{-}=\mu$ if $\theta=\mu^+$ and $\theta^{-}=\theta$ otherwise.}  Using the hypothesis that $\Kf$ is $\mu$-stable for every $\mu \in [\lambda_0, \lambda)$, one can build $\{ M_i : i < \kappa \}$ strictly increasing and continuous chain such that:

\begin{enumerate}
\item $M_{i+1}$ is $\|M_{i+1}\|$-universal over $M_i$.
\item $M_i \in \Kf_{\lambda_i}$.
\end{enumerate}

Let $M = \bigcup_{i< \kappa} M_i$. By construction $M$ is universal in $\Kf_\lambda$.\footnote{A similar construction is presented in \cite[3.18]{kuma}} Since $R$ is left perfect, $M$ is a cotorsion module. Then using Lemma \ref{uextension}, as in Lemma \ref{ccountablelim}, one can show that $\{ M^i : 0< i < \omega \}$ witnesses that $M^{(\aleph_0)}$ is a $(\lambda, \omega)$-limit model in $\Kf$. Hence $\Kf$ is $\lambda$-stable by Fact \ref{existence}.

$(2) \Rightarrow (3)$ Clear.

$ (3) \Rightarrow (4)$ We show that if $N$ is the $(\lambda, (|R| + \aleph_0)^+)$-limit model, then $N$ is $\Sigma$-cotorsion. This is enough by Lemma \ref{elem} and Fact \ref{sigmac}.(2).

Consider $\{N^{(\gamma)} : 0 < \gamma \leq  |R| + \aleph_0 \} \subseteq \Kf_\lambda$, we show by induction on $0< \gamma \leq |R|+\aleph_0$ that:

\begin{enumerate}
\item $N^{(\gamma)}$ is cotorsion.
\item $N^{(\gamma +1)}$ is universal over $N^{(\gamma)}$.
\end{enumerate}

Before we do the proof, observe that this is enough since by taking $\gamma =|R|+\aleph_0$ we have that $N^{(|R| + \aleph_0)}$ is cotorsion. Then by Fact \ref{sigmac} $N$ is $\Sigma$-cotorsion.

\underline{Base:} $N$ is cotorsion by Theorem \ref{bigpi2}, so (1) holds. Moreover, $N \oplus N$ is universal over $N$ by Lemma \ref{uextension}.

\underline{Induction step:} If $\gamma = \beta +1$, then $N^{(\beta +1)}$ is cotorsion because $N^{(\beta)}$ is cotorsion by induction hypothesis, $N$ is cotorsion and cotorsion modules are closed under finite direct sums. As for (2), this follows from Lemma \ref{uextension}.

If $\gamma$ is a limit ordinal, then consider $\{ N^{(\beta)} : 0<\beta < \gamma \}$. It is clear that it is an increasing and continuous chain in $\Kf_\lambda$ such that $\bigcup_{\beta < \gamma }N^{(\beta)} = N^{(\gamma)}$. Moreover, by induction hypothesis $N^{(\beta + 1)}$ is universal over $N^{(\beta)}$ for $\beta < \gamma$. Therefore, $\{ N^{(\beta)} : 0<\beta < \gamma \}$ witnesses that $N^{(\gamma)}$ is a $(\lambda, \gamma)$-limit model.  Then by uniqueness of limit models of size $\lambda$, $N^{(\gamma)}$ is isomorphic to $N$. We know that $N$ is cotorsion by Theorem \ref{bigpi2}, hence $N^{(\gamma)}$ is a cotorsion module. That $N^{(\gamma+ 1)}$ is universal over $N^{(\gamma)}$ then follows from Lemma \ref{uextension}.

$(4) \Rightarrow (1)$ Let $M \in \Kf$, by Fact \ref{xu2}  it is enough to show that $M$ is cotorsion. Let $\mu \geq \| M \| + (|R| + \aleph_0)^+$ such that $\Kf$ is $\mu$-stable, which exists by Fact \ref{ffact}.(3). Then fix $P \in \Kf$ a $(\mu, (|R| + \aleph_0)^+)$-limit model. Observe that by Fact \ref{euni} there is $f: M \to P$ a pure embedding. Since $M$ is $\Sigma$-cotorsion by hypothesis and $\Sigma$-cotorsion modules are closed under pure submodules by  Fact \ref{sigmac}. We conclude that $M$ is a cotorsion module. \end{proof}

\begin{remark} It was pointed out to us by Baldwin that in \cite[p. 159]{gar} the following is shown for right coherent rings: if $R$ is a left perfect ring, then every projective left $R$-module is totally transcendental. This can be used to show (1) implies (2) of the above theorem in the particular case when the ring is right coherent. This case is even more special than the one we consider in the next section (see Hypothesis \ref{hyp1}) since if a  ring $R$ is right coherent, then the class of flat left $R$-modules is first-order axiomatizable ( \cite[Theo. 4]{eksa}). 
\end{remark}

As a simple corollary we obtain a characterization of artinian rings via superstability.

\begin{cor}\label{art} For a ring $R$ the following are equivalent.
\begin{enumerate}
\item $R$ is right artinian.
\item $\Kf$ is superstable and the class of right $R$-modules with embeddings is superstable. 
\end{enumerate}
\end{cor}
\begin{proof}
It is known that a ring is right artinian if and only if it is left perfect and right noetherian (see for example \cite[Prop. 3]{eksa}). Moreover, $R$ is left perfect if and only if $\Kf$ is superstable by the theorem above. And $R$ is right noetherian if and only if the class of right $R$-modules with embeddings is superstable by \cite[3.12]{maz1}. \end{proof}

\section{A special case}

In this section we study $\Kf$ under Hypothesis \ref{hyp1} (see below). This allows us to characterize Galois-types, bound the values of $\theta_0$, $\theta_1$ and lower the bound in Theorem \ref{main3} where the tail of cardinals where uniqueness of limit models begins to $|R|+ \aleph_0$.

 We assume the next hypothesis throughout this section.

\begin{hypothesis}\label{hyp1}
The pure-injective envelope of every flat left $R$-module is flat.
\end{hypothesis}

These rings were characterized by Rothmaler in \cite{roth}. Every first-order axiomatizable class of flat modules satisfies this hypothesis since $M$ is an elementary substructure of its pure-injective envelope. Example 3.3 of \cite{roth} shows that there are rings satisfying Hypothesis \ref{hyp1} such that $\K^{\eff}$ is not first-order axiomatizable. This shows that the results in this section extend those obtained in \cite{kuma} for the class of flat modules.

One of the characterizations obtained in \cite{roth} that will be useful in this section is the following.

\begin{fact}[{\cite{roth}}]\label{requiv} For a ring $R$ the following are equivalent.
\begin{enumerate}
\item Hypothesis \ref{hyp1}, i.e., the pure-injective envelope of every flat left $R$-module is flat.
\item All flat cotorsion left $R$-modules are pure-injective.
\end{enumerate}

\end{fact}

Recall that $\phi$ is a positive primitive formula ($pp$-formula for short), if $\phi$ is an existentially quantified system of linear equations. For $M$ a module, $\bar{a} \in M^{<\omega}$ and $B \subseteq M$ we define the $pp$-type of $\bar{a}$ over $B$ in $M$, denoted by $pp(\bar{a}/ B, M)$, to be the set of $pp$-formulas $\phi(\bar{x}, \bar{b})$ such that $\bar{b} \in B$ and $M$ satisfies $\phi(\bar{a}, \bar{b})$. Recall the following result.

\begin{fact}[{\cite[3.6]{ziegler}}]\label{ziegler}
Let $M, N$ be pure-injective left $R$-modules, $A \subseteq M$ and $B \subseteq M$. If there is $f: A \to B$ a partial isomorphism\footnote{$f$ is a bijection between $A$ and $B$ and $f$ preserves $pp$-formulas.}, then there is $g: H^M(A)  \cong H^N(B)$ such that $g$ extends $f$.
\end{fact}

One of the missing pieces in the previous section is that we did not characterize Galois-types. The next lemma characterizes them under Hypothesis \ref{hyp1}. We obtain the same characterization as that of \cite[3.14]{kuma}, but with a conceptually different proof. The argument of  \cite[3.14]{kuma} can not be applied in this setting and vice versa.

\begin{lemma}\label{pp=gtp}
Let $M, N_1, N_2 \in \K^\eff$,  $M \leq_{pp} N_1, N_2$, $\bar{b}_{1}
\in  N_1^{<\omega}$  and $\bar{b}_{2} \in N_2^{<\omega}$. Then:
 \[ \gtp(\bar{b}_{1}/M; N_1) = \gtp(\bar{b}_{2}/M; N_2) \text{ if
and
only if } \pp(\bar{b}_{1}/M , N_1) = \pp(\bar{b}_{2}/M, N_2).\]
\end{lemma}
\begin{proof}
The forward direction is clear, so we only prove the backward direction.

 Assume  $\pp(\bar{b}_{1}/M , N_1) = \pp(\bar{b}_{2}/M, N_2)$, then by amalgamation there is $N \in \K^\eff$ and $f: N_1 \xrightarrow[M]{} N$ with $N \leq_{pp} N_2$. Since $\Kf$ is closed under pure-injective envelopes by Hypothesis \ref{hyp1}, $PE(N) \in \Kf$. Moreover, $N \leq_{pp} PE(N)$ so $\pp(f(\bar{b}_{1})/M , PE(N)) = \pp(\bar{b}_{2}/M, PE(N))$. Then by  Fact \ref{ziegler} there is \[g: H^{PE(N)}(M \cup \{f(\bar{b}_1) \}) \cong_{M} H^{PE(N)}(M \cup \{ \bar{b}_2 \}) \] with $g \circ f (\bar{b}_1)= \bar{b}_2$.

Since flat modules are closed under pure submodules, we have that $H^{PE(N)}(M \cup \{f(\bar{b}_1) \})$ and $H^{PE(N)}(M \cup \{ \bar{b}_2 \})$ are flat. Then applying amalgamation a couple of times we get the desired result.
\end{proof}

As  a corollary we obtain that $\theta_0=\aleph_0$, this improves the results of \cite[\S 6]{lrv} (Fact \ref{ffact}.(4)) for classes of flat modules with pure embeddings under Hypothesis \ref{hyp1}.

\begin{cor}
$\Kf$ is $(<\aleph_0)$-tame.
\end{cor}

As in \cite[3.16-3.19]{kuma} one can obtain the following results. 

 \begin{lemma}\label{g-st}\
\begin{enumerate}
\item If $\lambda^{|R| + \aleph_0}=\lambda$, then $\Kf$ is $\lambda$-stable.
\item  If $\lambda^{|R| + \aleph_0}=\lambda$ or $\forall \mu < \lambda( \mu^{|R| + \aleph_0} <
\lambda)$, then $\Kf_\lambda$ has a universal model
\end{enumerate}
\end{lemma}
\begin{proof}[Proof sketch]\
\begin{enumerate}
\item Let $M \in \Kf_\lambda$ and $\{\gtp(a_i/ M; N) : i < \alpha \}$ be an enumeration without repetitions of $\gS(M)$. Then define $\Phi: gS(M) \to S^{Th(N)}_{pp}(M)$ by $\phi(\gtp(a_i/ M; N))=pp(a_i/M, N)$.  Using that $\lambda^{|R| + \aleph_0}=\lambda$ and $pp$-quantifier elimination the result follows.
\item If $\lambda^{|R| + \aleph_0}=\lambda$, then there are limit models of cardinality $\lambda$ and limit models are universal models. If $\forall \mu < \lambda( \mu^{|R| + \aleph_0} < \lambda)$, the argument is similar to the induction step of (1) implies (2) of Theorem \ref{main3}.
\end{enumerate} \end{proof}

\begin{remark} Recall that \cite[1.2]{sh820} asserts that there is a universal group of size $\lambda$ in the class of torsion-free abelian groups with pure embeddings if $\lambda^{\aleph_0}=\lambda$ or $\forall \mu < \lambda( \mu^{ \aleph_0} <
\lambda)$. Observe that the class of torsion-free abelian groups is the class of flat $\mathbb{Z}$-modules and it satisfies Hypothesis \ref{hyp1}. Therefore, the above lemma generalizes  \cite[1.2]{sh820} to classes of flat modules not axiomatizable in first-order logic.  \end{remark}

\begin{remark}
Observe that the above lemma bounds $\theta_1$ by $|R|+\aleph_0$.
\end{remark}

In this case we get that \emph{long} limit models are not only cotorsion modules, but they are pure-injective modules.

\begin{lemma}\label{bigpi}
Assume $\lambda\geq (|R| + \aleph_0)^+$. If $M$ is a $(\lambda,
\alpha)$-limit model in $\Kf$ and $\cof(\alpha)\geq (|R| + \aleph_0)^+$, then $M$ is
pure-injective.
\end{lemma}
\begin{proof}
By Theorem \ref{bigpi2} $M$ is a cotorsion module. Then by Hypothesis \ref{hyp1} and Fact \ref{requiv} it follows that $M$ is pure-injective.  \end{proof}

It is not a coincidence that we had to use Hypothesis \ref{hyp1} to obtain the above result. The next result shows that both notions are equivalent.

\begin{theorem}\label{newroth} For a ring $R$ the following are equivalent.
\begin{enumerate}
\item Every $(\lambda, \alpha)$-limit model in $\Kf$ with $\lambda \geq  (|R| + \aleph_0)^+$ and $\cof(\alpha)\geq (|R| + \aleph_0)^+$ is  pure-injective.
\item Hypothesis \ref{hyp1}, i.e., the pure-injective envelope of every flat left $R$-module is flat.
\end{enumerate}
\end{theorem}
\begin{proof}
The backward direction is Lemma \ref{bigpi}, so we show the forward direction. Let $M \in \Kf$. Pick $\lambda \geq \|M \| + (|R| + \aleph_0)^+$ such that $\Kf$ is $\lambda$-stable, this is possible by Fact \ref{ffact}.(3). Then by Fact \ref{existence} there is $N$ a $(\lambda, (|R| + \aleph_0)^+)$-limit model. From the assumption we have that $N$ is pure-injective and since there is $f: M \to N$ a pure embedding, it follows that $PE(M) \cong PE(f[M]) \leq_{pp} N$. Since $\Kf$ is closed under pure submodules, we conclude that $PE(M) \in \Kf$.  \end{proof}

\begin{remark}
Since Hypothesis \ref{hyp1} is one of the equivalent assertions of the main theorem of \cite[2.3]{roth},  the above theorem gives a new characterization of the rings studied in \cite{roth}.
\end{remark}

To finish this section we  show that under Hypothesis \ref{hyp1}, one can lower the bound where the tail of uniqueness of limit cardinals begins to $|R|+\aleph_0$.

\begin{theorem}\label{main1} For a ring $R$ satisfying Hypothesis \ref{hyp1} the following are equivalent.
\begin{enumerate}
\item $R$ is left perfect.
\item The class of flat left $R$-modules with pure embeddings is superstable.
\item There is a $\lambda \geq (|R| + \aleph_0)^+$ such that the class of flat left $R$-modules with pure embeddings has uniqueness of limit models of cardinality $\lambda$.
\item Every limit model in the class of flat left $R$-modules with pure embeddings is $\Sigma$-pure-injective.
\item For every $\lambda \geq |R| + \aleph_0$,  the class of flat left $R$-modules with pure embeddings has uniqueness of limit models of cardinality $\lambda$.

\end{enumerate}
\end{theorem}
\begin{proof} $(1) \Leftrightarrow (2) \Leftrightarrow (3) \Leftrightarrow (4)$ Follow from Theorem \ref{main3} and Fact \ref{requiv}.

$(5) \Rightarrow (2)$ Clear.

$(1) \Rightarrow (5)$ Since every limit model is a cotorsion module. Then by Corollary \ref{easy} there is at most one $\lambda$-limit model for each $\lambda$ up to isomorphisms. Hence to finish the proof, it is enough to show that for every $\lambda \geq |R| + \aleph_0$, $\Kf$ is $\lambda$-stable.

Let $\lambda \geq |R| + \aleph_0$ and $M \in \Kf_\lambda$. Let $\{\gtp(a_i/ M; N) : i < \alpha \}$ be an enumeration without repetitions of $\gS(M)$. We can assume that they are all realized in a fixed $N$ by amalgamation. Now, consider $\Phi: gS(M) \to S^{Th(N)}_{pp}(M)$ given by $\Phi(\gtp(a_i/ M; N))=pp(a_i/M, N)$. By  Lemma \ref{pp=gtp} it follows that $\Phi$ is a well-defined injective function. Since $N$ is $\Sigma$-pure-injective by (1) and Hypothesis \ref{hyp1}, $Th(N)$ is totally transcendental (see for example \cite[3.2]{prest}). In particular, since complete theories of modules have $pp$-quantifier elimination we can conclude that $|S^{Th(N)}_{pp}(M)|=|S^{Th(N)}(M)| \leq \lambda$. Therefore, $|\gS(M)| \leq \lambda$.

 \end{proof}

\begin{remark} Recall that $\mathbb{Z}$ is not a perfect ring. Then by condition five of the above theorem we have that  the class of torsion-free abelian groups with pure embeddings does not have uniqueness of limit models in any uncountable cardinal. This was shown in \cite[4.26]{maz} using AEC methods and in \cite[4.15]{kuma} using group theoretic methods.
\end{remark}


\end{document}